\documentclass{amsart}

\usepackage{amsmath,amssymb,amsthm}

\newtheorem{prop}{Proposition}
\newtheorem{thm}[prop]{Theorem}
\newtheorem{lem}[prop]{Lemma}

\theoremstyle{definition}

\newtheorem{rem}[prop]{Remark}

\def\co{\colon\thinspace}

\newcommand{\N}{\mathbb{N}}

\newcommand{\PDM}{\mathrm{PD}_M}
\newcommand{\PDS}{\mathrm{PD}_{\Sigma}}
\newcommand{\PKW}{\Phi_{\mathrm{KW}}}

\newcommand{\R}{\mathbb R}

\newcommand{\Z}{\mathbb Z}

\begin{document}

\author[H. Geiges]{Hansj\"org Geiges}
\address{Mathematisches Institut, Universit\"at zu K\"oln,
Weyertal 86--90, 50931 K\"oln, Germany}
\email{geiges@math.uni-koeln.de}
\author[K. Sporbeck]{Kevin Sporbeck}
\author[K. Zehmisch]{Kai Zehmisch}
\address{Fakult\"at f\"ur Mathematik, Ruhr-Universit\"at Bochum,
Universit\"atsstra{\ss}e 150, 44780 Bochum, Germany}
\email{kevin.sporbeck@rub.de, kai.zehmisch@rub.de}

\title[Subcritical polarisations]{Subcritical polarisations of
symplectic manifolds have degree one}

\date{}

\begin{abstract}
We show that if the complement of a Donaldson hypersurface
in a closed, integral symplectic manifold has
the homology of a subcritical Stein manifold, then
the hypersurface is of degree one. In particular, this
demonstrates a conjecture by Biran and Cieliebak on subcritical
polarisations of symplectic manifolds. Our proof is based on a
simple homological argument using ideas of Kulkarni--Wood.
\end{abstract}

\subjclass[2020]{53D35, 57R17, 57R19, 57R95}
\thanks{This research is part of a project in the SFB/TRR 191
\textit{Symplectic Structures in Geometry, Algebra and Dynamics}, 
funded by the DFG (Project-ID 281071066 -- TRR 191)}

\maketitle


\section{Donaldson hypersurfaces and symplectic polarisations}
Let $(M,\omega)$ be a closed, connected, integral symplectic manifold,
that is, the de Rham cohomology class $[\omega]_{\mathrm{dR}}$ lies in
the image of the homomorphism
$H^2(M)\rightarrow H^2_{\mathrm{dR}}(M)=H^2(M;\R)$ induced by the
inclusion $\Z\rightarrow\R$. The cohomology classes
in $H^2(M)$ mapping to $[\omega]_{\mathrm{dR}}$ are called
\emph{integral lifts}, and by abuse of notation we shall
write $[\omega]$ for any such lift.
Following McDuff and Salamon~\cite[Section~14.5]{mcsa17}, we call
a codimension~$2$ symplectic submanifold $\Sigma\subset M$
a \textbf{Donaldson hypersurface} if it is Poincar\'e dual
to $d[\omega]\in H^2(M)$ for some integral lift $[\omega]$
and some (necessarily positive) integer~$d$. Donaldson~\cite{dona96}
has established the existence of such hypersurfaces for any
sufficiently large~$d$.

The pair $(M,\Sigma)$ is called a \textbf{polarisation} of $(M,\omega)$,
and the number $d\in\N$, the \textbf{degree} of the polarisation.
Biran and Cieliebak~\cite{bici01} studied these polarisations in the
K\"ahler case, where $\omega$ admits a compatible
\emph{integrable} almost complex structure~$J$. In that setting,
the complement $(M\setminus\Sigma,J)$ admits in a natural way
the structure of a Stein manifold.

As shown recently by Giroux~\cite{giro17}, building on work
of Cieliebak--Eliashberg, even in the non-K\"ahler
case the complement of a symplectic hypersurface
$\Sigma\subset M$ found by Donaldson's
construction admits the structure of a Stein manifold. Here, of course,
the complex structure on $M\setminus\Sigma$ does not, in general,
extend over~$\Sigma$. Complements of Donaldson hypersurfaces
are also studied in~\cite{dili19}.
\section{Subcritical polarisations}
The focus of Biran and Cieliebak~\cite{bici01} lay on \emph{subcritical}
polarisations of K\"ahler manifolds, which means that $(M\setminus\Sigma,J)$
admits a plurisubharmonic Morse function $\varphi$ all of whose critical
points have, for $\dim M=2n$, index less than~$n$. (They also assumed that
$\varphi$ coincides with the plurisubharmonic function
defining the natural Stein structure outside a compact set
containing all critical points of $\varphi$.)

More generally, McDuff and Salamon~\cite[p.~504]{mcsa17}
propose the study of polarisations $(M,\Sigma)$ where
the complement $M\setminus\Sigma$ is homotopy equivalent
to a subcritical Stein manifold (of finite type).
We relax this condition a little
further and call $(M,\Sigma)$ \textbf{homologically subcritical}
if $M\setminus\Sigma$ has the \emph{homology} of a subcritical Stein
manifold, that is, of a CW-complex containing
finitely many cells up to dimension at most $n-1$.
This means that there is some $\ell\leq n-1$
such that $H_k(M\setminus\Sigma)$
vanishes for $k\geq\ell+1$ and $H_{\ell}(M\setminus\Sigma)$ is torsion-free.

Motivated by the many examples they could construct,
Biran and Cieliebak~\cite[p.~751]{bici01} conjectured that
subcritical polarisations necessarily have degree~$1$.
They suggested an approach to this conjecture using either
symplectic or contact homology. A rough sketch of a proof
along these lines, in the language of symplectic
field theory, was given by
Eliashberg--Givental--Hofer~\cite[p.~661]{egh00}.
A missing assumption $c_1(M\setminus\Sigma)=0$ of that argument
and a few more details --- still short of a complete proof ---
were added by J.~He~\cite[Proposition~4.2]{he13}, who
appeals to Gromov--Witten theory and polyfolds.

Here is our main result, which entails the conjecture of Biran--Cieliebak.

\begin{thm}
\label{thm:main}
Let $(M,\omega)$ be a closed, integral symplectic manifold,
and $\Sigma\subset M$ a compact symplectic submanifold of codimension~$2$,
Poincar\'e dual to the integral cohomology class $d[\omega]$
for some (positive) integer~$d$. If $(M,\Sigma)$ is homologically subcritical,
then $d[\omega]/\mathrm{torsion}$ is indivisible in
$H^2(M)/\mathrm{torsion}$. In particular, $d=1$.
\end{thm}

Our proof is devoid of any sophisticated machinery.
The assumption on $(M,\Sigma)$ to be homologically subcritical
guarantees the surjectivity of a certain
homomorphism in homology described by
Kulkarni and Wood~\cite{kuwo80}; this implies the claimed indivisibility.
\section{The Kulkarni--Wood homomorphism}
We consider a pair $(M,\Sigma)$ consisting of a closed, connected,
oriented manifold $M$ of dimension~$2n$, and a compact, oriented
hypersurface $\Sigma\subset M$ of codimension~$2$. No symplectic
assumptions are required in this section.

Write $i\co\Sigma\rightarrow M$ for the inclusion map.
The Poincar\'e duality isomorphisms on $M$ and $\Sigma$
from cohomology to homology,
given by capping with the fundamental class, are denoted by
$\PDM$ and $\PDS$, respectively.

In their study of the topology of complex hypersurfaces,
Kulkarni and Wood~\cite{kuwo80} used the following composition,
which we call the \emph{Kulkarni--Wood homomorphism}:
\[ \PKW\co H^k(M) \stackrel{i^*}{\longrightarrow}
H^k(\Sigma) \stackrel{\PDS}{\longrightarrow}
H_{2n-2-k}(\Sigma) \stackrel{i_*}{\longrightarrow}
H_{2n-2-k}(M) \stackrel{\PDM^{-1}}{\longrightarrow}
H^{k+2}(M).\] 

\begin{lem}
\label{lem:KW}
The Kulkarni--Wood homomorphism equals the cup product with
the cohomology class $\sigma:=\PDM^{-1}(i_*[\Sigma])\in H^2(M)$.
\end{lem}

\begin{proof}
For $\alpha\in H^k(M)$ we compute
\[ \begin{array}{rcl}
\PKW(\alpha)
& = & \PDM^{-1}\,i_*\,\PDS\,i^*\alpha
\; = \; \PDM^{-1}\,i_*\bigl(i^*\alpha\cap[\Sigma]\bigr)\\[1mm]
& = & \PDM^{-1}\bigl(\alpha\cap i_*[\Sigma]\bigr)
\; = \; \PDM^{-1}\bigl(\alpha\cap\PDM(\sigma)\bigr)\\[1mm]
& = & \PDM^{-1}\bigl(\alpha\cap(\sigma\cap[M])\bigr) 
\; = \; \PDM^{-1}\bigl((\alpha\cup\sigma)\cap[M]\bigr)\\[1mm]
& = & \alpha\cup\sigma.\qedhere
\end{array}\]
\renewcommand{\qed}{}
\end{proof}

\begin{lem}
\label{lem:surjective}
If the complement $M\setminus\Sigma$ has the homology type
of a CW-complex of dimension $\ell$ for some $\ell\leq n-1$,
then $\PKW\co H^k(M)\rightarrow H^{k+2}(M)$ is surjective 
in the range $\ell-1\leq k\leq 2n-\ell-2$.
\end{lem}

\begin{proof}
Write $\nu\Sigma$ for an open tubular neighbourhood
of $\Sigma$ in~$M$.
By homotopy, excision, duality, and the universal
coefficient theorem we have
\[ \begin{array}{rcl}
H_k(M,\Sigma) & \cong & H_k(M,\nu\Sigma) \; \cong \;
                    H_k(M\setminus\nu\Sigma,\partial(\nu\Sigma))\\[1mm]
 & \cong & H^{2n-k}(M\setminus\nu\Sigma) \; \cong \;
           FH_{2n-k}(M\setminus\Sigma)
           \oplus TH_{2n-k-1}(M\setminus\Sigma),
\end{array}\]
where $F$ and $T$ denotes the free and the torsion part,
respectively. This vanishes for $2n-k-1\geq\ell$, that is,
for $k\leq 2n-\ell-1$. It follows that the homomorphism
$i_*\co H_{2n-2-k}(\Sigma)\rightarrow H_{2n-2-k}(M)$ is surjective
for $2n-2-k\leq 2n-\ell-1$, or $k\geq\ell-1$.

Similarly (or directly by Poincar\'e--Lefschetz duality) we have
\[ H^k(M,\Sigma)\cong H_{2n-k}(M\setminus\Sigma),\]
which vanishes for $2n-k\geq \ell+1$, that is, for $k\leq 2n-\ell-1$.
Hence, the homomorphism $i^*\co H^k(M)\rightarrow H^k(\Sigma)$
is surjective for $k+1\leq 2n-\ell-1$, that is, for $k\leq 2n-\ell-2$.
\end{proof}
\section{Proof of Theorem~\ref{thm:main}}
Under the assumptions of Theorem~\ref{thm:main}, the
homomorphism $\PKW\co H^k(M)\rightarrow H^{k+2}(M)$
is surjective at least in the range $n-2\leq k\leq n-1$;
simply set $\ell=n-1$ in Lemma~\ref{lem:surjective}.
Thus, we can pick an even number $k=2m$ in this range.
The free part of $H^{2m+2}(M)$ is non-trivial, since
this cohomology group contains the element $[\omega]^{m+1}$
of infinite order.

On the other hand, $\PKW$ is given by the cup product
with $d[\omega]$, as shown in Lemma~\ref{lem:KW}.
If $d[\omega]/\mathrm{torsion}$ were divisible, so would
be all elements in the image of $\PKW$ in $H^{2m+2}(M)/\mathrm{torsion}$,
and $\PKW$ would not be surjective.

\begin{rem}
The real Euler class of the circle bundle $\partial(\nu\Sigma)$
equals $d[\omega]_{\mathrm{dR}}$, and the natural Boothby--Wang
contact structure on this bundle has an exact convex filling
by the complement $M\setminus\nu\Sigma$, see
\cite[Lemma~3]{gest10}, \cite[Proposition~5]{giro17}
and \cite[Lemma~2.2]{dili19}.
With \cite[Theorem~1.2]{bgz19} the condition `homologically
subcritical' of Theorem~\ref{thm:main} may be replaced by
assuming the existence of \emph{some} subcritical Stein filling
of this Boothby--Wang contact structure.
\end{rem}

\end{document}